\documentclass[11pt]{amsart}

\usepackage{amsmath,amssymb,amsthm,amscd,latexsym}
\usepackage[frame,cmtip,arrow,matrix,line,graph,curve]{xy}
\usepackage{graphpap, color}
\usepackage[mathscr]{eucal}
\usepackage{color}

\newtheorem{prop}{Proposition}
\newtheorem{thm}[prop]{Theorem}
\newtheorem{lem}[prop]{Lemma}
\newtheorem{cor}[prop]{Corollary}
\newtheorem{rem}[prop]{Remark}

\theoremstyle{definition}

\begin{document}
\title{A proof of the Alexanderov's uniqueness theorem for convex surfaces in $\mathbb R^3$}
\author{Pengfei Guan}
\address{Department of Mathematics and Statistics\\ McGill University \\ Montreal, Canada}
\email{guan@math.mcgill.ca}
\author{Zhizhang Wang}
\address{Department of Mathematics\\ Fudan University \\ Shanghai, China}
\email{zwang@math.mcgill.ca}
\author{Xiangwen Zhang}
\address{Department of Mathematics\\ Columbia University \\ New York, US}
\email{xzhang@math.columbia.edu}
\date{\today}
\thanks{Research of the first author was supported in part by an NSERC Discovery Grant.}
\begin{abstract}
  We give a new proof of a classical uniqueness theorem of Alexandrov \cite{Alex56} using the weak uniqueness continuation theorem of Bers-Nirenberg \cite{BN2}. We prove a version of this theorem with the minimal regularity assumption: the spherical hessians of the corresponding convex bodies as Radon measures are nonsingular.
\end{abstract}

\maketitle

We give a new proof of the following uniqueness theorem of Alexandrov, using the Weak Unique Continuation Theorem of Bers-Nirenberg \cite{BN2}. 
\begin{thm}[Theorem 9 in \cite{Alex56}]\label{alex0} Suppose $M_1$ and $M_2$ are two closed strictly convex $C^2$ surfaces in $\mathbb R^3$, suppose $f(y_1,y_2)\in C^1$ is a function such that $\frac{\partial f}{\partial y_1} \frac{\partial f}{\partial y_2}>0$. Denote $\kappa_1\ge \kappa_2$ the principal curvatures of surfaces, and denote $\nu_{M_1}$ and $\nu_{M_2}$ the Gauss maps of $M_1$ and $M_2$ respectively. If
\begin{equation}\label{1}f(\kappa_1(\nu^{-1}_{M_1}(x), \kappa_2(\nu^{-1}_{M_1}(x))=f(\kappa_1(\nu^{-1}_{M_2}(x), \kappa_2(\nu^{-1}_{M_2}(x)), \quad \forall x\in \mathbb S^2,\end{equation}
then $M_1$ is equal to $M_2$ up to a translation.\end{thm}

This classical result was first proved for analytical surfaces by Alexandrov in \cite{Alex39}, for $C^4$ surfaces by Pogorelov in \cite{P1}, and Hartman-Wintner \cite{HW} reduced regularity to $C^3$, see also \cite{P2}. Pogorelov \cite{P3, P4} published certain uniqueness results for $C^2$ surfaces, these general results would imply Theorem \ref{alex0} in $C^2$ case. It was pointed out in \cite{Pa} that the proof of Pogorelov is erroneous, it contains an uncorrectable mistake (see page 301-302 in \cite{Pa}). There is a counter-example of Martinez-Maure \cite{Ma} (see also \cite{Pa}) to the main claims in \cite{P3,P4}. The results by Han-Nadirashvili-Yuan \cite{HNY} imply two proofs of Theorem \ref{alex0}, one for $C^2$ surfaces and another for $C^{2,\alpha}$ surfaces. The problem is often reduced to a uniqueness problem for linear elliptic equations in appropriate settings, either on $\mathbb S^2$ or in $\mathbb R^3$, we refer \cite{Alex56, P2}. Here we will concentrate on the corresponding equation on $\mathbb S^2$, as in \cite{GG}. The advantage in this setting is that it is globally defined.

If $M$ is a strictly convex surface with support function $u$, then the principal curvatures at $\nu^{-1}(x)$ are the reciprocals of the principal radii $\lambda_1,\lambda_2$ of $M$,  which are the eigenvalues of spherical Hessian $W_u(x)=(u_{ij}(x)+u(x)\delta_{ij})$ where $u_{ij}$ are the covariant derivatives with respect to any given local orthonormal frame on $\mathbb S^2$. Set
\begin{eqnarray}\label{Ff}
 \tilde{F}(W_u)=: f(\frac1{\lambda_1(W_u)}, \frac1{\lambda_2(W_u)})= f(\kappa_1, \kappa_2).\end{eqnarray}
In view of Lemma 1 in \cite{Alex57}, if $f$ satisfies the conditions in Theorem \ref{alex0}, then
$\tilde{F}^{ij}=\frac{\partial \tilde{F}}{\partial w_{ij}}\in L^{\infty}$ is uniformly elliptic. In the case $n=2$, it can be read off from the explicit formulas
\begin{eqnarray*}
\lambda_1= \frac{\sigma_1(W_u)- \sqrt{\sigma_1(W_u)^2-4\sigma_2(W_u)}}{2}, \quad
\lambda_2=  \frac{\sigma_1(W_u)+ \sqrt{\sigma_1(W_u)^2-4\sigma_2(W_u)}}{2}.
\end{eqnarray*}
As noted by Alexanderov in \cite{Alex57}, $\tilde{F}^{ij}$ in general is not continuous if $f(y_1,y_2)$ is not symmetric (even $f$ is analytic).

We want to address when Theorem \ref{alex0} remains true for convex bodies in $\mathbb R^3$ with weakened regularity assumption. In the Bruun-Minkowski theory, the uniqueness of Alexandrov-Fenchel-Jessen \cite{Alex37, Alex38, FJ} states that, if two bounded convex bodies in $\mathbb R^{n+1}$ have the same $k$th area measures on $\mathbb S^n$, then these two bodies are the same up to a rigidity motion in $\mathbb R^{n+1}$. Though for a general convex body, the principal curvatures of its boundary may not be defined. But one can always define the support function $u$, which is a function on $\mathbb S^2$. By the convexity, then $W_u=(u_{ij}+u\Delta_{ij})$ is a Radon measure on $\mathbb S^2$. Also, by Alexandrov's theorem for the differentiability of convex functions, $W_{u}$ is defined for almost every point $x\in \mathbb S^2$. Denote $\mathcal N$ to be the space of all positive definite $2\times 2$ matrices, and let $F$ be a function defined on $\mathcal N$. For a support function $u$ of a bounded convex body $\Omega_u$, $F(W_u)$ is defined for $a.e.\  x\in \mathbb S^2$. For fixed support functions $u^l$ of $\Omega_{u^l}, l=1,2$, there is $\Omega\subset \mathbb S^2$ with $|\mathbb S^2\setminus \Omega|=0$ such that $W_{u^1},
W_{u^2}$ are pointwise finite in $\Omega$. Set $P_{u^1, u^2}=\{W\in \mathcal{N} | \exists x\in \Omega,  W=W_{u^1}(x), or W=W_{u^2}(x) \}$, let $\mathcal{P}_{u^1, u^2}$ be the convex hull of $P_{u^1, u^2}$ in $\mathcal{N}$.

We establish the following slightly more general version of Theorem \ref{alex0}.
\begin{thm}\label{alexB}
Suppose $\Omega_1$ and $\Omega_2$ are two bounded convex bodies in $\mathbb R^3$. Let $u^l, l=1,2$ be the corresponding supporting functions respectively. Suppose the spherical Hessians $W_{u^l}= (u^l_{ij}+\delta_{ij}u^l)$ (in the weak sense) are two non-singular Radon measures. Let $F: \mathcal N\rightarrow \mathbb R$ be a $C^{0, 1}$ function such that
\[\Lambda I\ge (F^{ij})(W):=(\frac{\partial F}{\partial W_{ij}})(W)\ge \lambda I >0, \quad \forall W\in \mathcal{P_{u^1, u^2}},\]
for some positive constants $\Lambda, \lambda$. If
\begin{equation}\label{F}
F(W_{u^1})= F(W_{u^2}),
\end{equation}
at almost every parallel normal  $ x\in \mathbb S^2$,
then $\Omega_1$ is equal to $\Omega_2$ up to a translation.
\end{thm}

\medskip

Suppose $u^1, u^2$ are the support functions of two convex bodies $\Omega_1,\Omega_2$ respectively, and suppose $W_{u^^l}, l=1,2$ are defined and they satisfy equation (\ref{F}) at some point $x\in \mathbb S^2$. Then, for $u=u^1-u^2$, $W_u(x)$ satisfies equation
\begin{equation}\label{Fu}
 F^{ij}(x) (W_u(x))=0, \end{equation}
with $F^{ij}(x)=\int_0^1 \frac{\partial F}{\partial W_{ij}}(tW_{u^1}(x)+(1-t)W_{u^2}(x))dt$.
By the convexity, $W_{u^l}, l=1,2$ exist almost everywhere on $\mathbb S^2$. If they satisfy equation (\ref{F}) almost everywhere, equation (\ref{Fu}) is verified almost everywhere. Note that $u$ may {\it not be a solution (even in a weak sense) of partial differential equation (\ref{Fu})}. The classical elliptic theory (e.g., \cite{M, N2, BN2}) requires $u\in W^{2,2}$ in order to make sense of $u$ as a weak solution of (\ref{Fu}). A main step in the proof of Theorem \ref{alexB} is to show that with the assumptions in the theorem, $u=u^1-u^2$ is indeed in $W^{2,2}(\mathbb S^2)$. The proof will appear in the last part of the paper.

\medskip

Let's now focus on $W^{2,2}$ solutions of differential equation (\ref{Fu}), with general uniformly elliptic condition on tensor $F^{ij}$ on $\mathbb S^2$:
\begin{equation}\label{elliptic}
\lambda |\xi|^2\le F^{ij}(x)\xi_i\xi_j\le \Lambda |\xi|^2, \ \forall x\in \mathbb S^2, \xi \in \mathbb R^2,
\end{equation}
for some positive numbers $\lambda, \Lambda$.
The aforementioned proofs of Theorem \ref{alex0} (\cite{P1, HW, P2, HNY}) all reduce to the statement that any solution of (\ref{elliptic}) is a linear function, under various regularity assumptions on $F^{ij}$ and $u$. Equation (\ref{Fu}) is also related to minimal cone equation in $\mathbb R^3$ (\cite{HNY}). The following result was proved in \cite{HNY}.

\begin{thm}[Theorem 1.1 in \cite{HNY}]\label{alex}
Suppose $F^{ij}(x)\in L^{\infty}(\mathbb S^2)$ satisfies (\ref{elliptic}), suppose $u\in W^{2,2}(\mathbb S^2)$ is a solution of
(\ref{Fu}). Then, $u(x)=a_1x_1+a_2x_2+a_3x_3$ for some $a_i\in \mathbb R$. \end{thm}

There the original statement in \cite{HNY} is for $1$-homogeneous $W^{2,2}_{loc}(\mathbb R^3)$ solution $v$ of equation
\begin{equation}\label{alex1} \sum_{i,j=1}^3a^{ij}(X)v_{ij}(X)=0.\end{equation}
These two statements are equivalent. To see this, set $u(x)=\frac{v(X)}{|X|}$ with $x=\frac{X}{|X|}$. By the homogeneity assumption, the radial direction corresponds to null eigenvalue of $\nabla^2 v$, the other two eigenvalues coincide the eigenvaules of the spherical Hessian of $W=(u_{ij}+u\delta_{ij})$. $v(X)\in W^{2,2}_{loc}(\mathbb R^3)$ is a solution to (\ref{alex1}) if and only if $u\in  W^{2,2}(\mathbb S^2)$ is a solution to (\ref{Fu}) with $F^{ij}(x)=\langle e_i, A e_j\rangle $, where $A=(a^{ij}(\frac{X}{|X|}))$ and $(e_1,e_2)$ is any orthonormal frame on $\mathbb S^2$.

The proof in \cite{HNY} uses gradient maps and support planes introduced by Alexandrov, as in \cite{Alex39, P1, P2}. We give a different proof of Theorem \ref{alex} using the maximum principle for smooth solutions and the unique continuation theorem of Bers-Nirenberg \cite{BN2}, working purely on solutions of equation (\ref{Fu}) on $\mathbb S^2$.

\medskip

Note that $F$ in Theorem \ref{alexB} (and Theorem \ref{alex0}) is not assumed to be symmetric. The weak assumption $F^{ij}\in L^{\infty}$ is needed to deal with this case.
This assumption also fits well with the weak unique continuation theorem of Bers-Nirenberg. This beautiful result of Bers-Nirenberg will be used in a crucial way in our proof. If $u\in W^{2,2}(\mathbb S^2)$, $u\in C^{\alpha}(\mathbb S^2)$ for some $0<\alpha<1$ by the Sobolev embedding theorem. Equation (\ref{Fu}) and $C^{1,\alpha}$ estimates for $2$-d linear elliptic PDE (e.g., \cite{M, N2, BN2}) imply that $u$ is in $C^{1,\alpha}(\mathbb S^2)$ for some $\alpha>0$ depending only on $\|u\|_{C^0}$ and the ellipticity constants of $F^{ij}$. This fact will be assumed in the rest of the paper.

The following lemma is elementary.

\begin{lem}\label{det}
Suppose $F^{ij}\in L^{\infty}(\mathbb S^2)$ satisfies (\ref{elliptic}), suppose at some point $x\in \mathbb S^2$, $W_u(x)=(u_{ij}(x)+u(x)\delta_{ij})$ satisfies (\ref{Fu}). Then,
\[
|W_u|^2(x) \leq -\frac{2\Lambda}{\lambda}\det W_u(x).
\]
\end{lem}
\begin{proof}
At $x$, by equation (\ref{Fu}),
\begin{eqnarray}\label{detW}
\det W_u =  -\frac{1}{F^{22}} \Big( F^{11} W_{11}^2 + 2F^{12} W_{11}W_{12} + F^{22}W_{12}^2\Big)\le -\frac{\lambda}{\Lambda} \big(W_{11}^2+ W_{12}^2 \big),
\end{eqnarray}
and similarly, $\det W_u  \leq -\frac{\lambda}{\Lambda} \big(W_{22}^2+ W_{21}^2 \big)$. Thus,
\begin{eqnarray}\label{W22}
 \big(W_{11}^2+ W_{12}^2 + W_{21}^2 +W_{22}^2\big)\leq  -\frac{2\Lambda}{\lambda} \det W_u.
\end{eqnarray}
\end{proof}

\medskip

For each $u\in C^{1}(\mathbb S^2)$, set
$X_u= \sum_i u_i e_{i} +u e_{n+1}$.
For any unit vector $E$ in $\mathbb R^3$, define
\begin{eqnarray}\label{phie}
\phi_E(x)=\langle E, X_u(x)\rangle, \quad  \mbox{and} \quad \rho_u(x)= | X_u(x)|^2,
\end{eqnarray}
where $\langle , \rangle $ is the standard inner product in $\mathbb R^3$. The function $\rho$  was introduced by Weyl in his study of Weyl's problem \cite{W}. It played important role in Nirenberg's solution of the Weyl's problem in \cite{N1}. Our basic observation is that there is a maximum principle for $\rho_u$ and $\phi_E$.

\begin{lem}\label{lem1}   Suppose $U\subset \mathbb S^2$ is an open set, $F^{ij}\in C^1(U)$ is a tensor in $U$ and $u\in C^3(U)$ satisfies equation (\ref{Fu}), then there are two constants $C_1,C_2$ depending only on the $C^1$-norm of $F^{ij}$ such that
\begin{eqnarray}\label{ineq1}
F^{ij}(\rho_u)_{ij}\geq -C_1|\nabla \rho_u| ,
\quad
F^{ij}(\phi_E)_{ij}\geq - C_2|\nabla \phi_E| \quad \mbox{in $U$}.
\end{eqnarray}
\end{lem}
\begin{proof}
Pick any orthonormal frame $e_1,e_2$,  we have
\begin{eqnarray}
(X_u)_i =W_{ij}e_j, \quad
(X_u)_{ij}=W_{ijk}e_k-W_{ij}\vec{x}.
\end{eqnarray}
By Codazzi property of $W$ and (\ref{Fu}),
\begin{eqnarray}\label{3}
\frac{1}{2} F^{ij}(\rho_u)_{ij}= \langle X_u,F^{ij}W_{ijk}e_k \rangle+F^{ij}W_{ik}W_{kj}
= -u_kF^{ij}_{,k}W_{ij}+F^{ij}W_{ik}W_{kj}\nonumber.
\end{eqnarray}

On the other hand, $\nabla \rho_u = 2W\cdot (\nabla u) $.
At the non-degenerate points (i.e., $\det W\neq 0$),
$\nabla u =\frac{1}{2} W^{-1}\cdot \nabla \rho_u$,
where $W^{-1}$ denotes the inverse matrix of $W$. Now,
\begin{eqnarray}\label{4}
2u_kF^{ij}_{,k}W_{ij}=W^{kl}(\rho_u)_lF^{ij}_{,k}W_{ij}=(\rho_u)_l F^{ij}_{,k}\frac{A^{kl}W_{ij}}{\det W}.
\end{eqnarray}
where $A^{kl}$ denote the co-factor of $W_{kl}$.

The first inequality in (\ref{ineq1}) follows (\ref{W22}) and (\ref{4}).

The proof for $\phi_E$ follows the same argument and the following facts:
\begin{eqnarray*}
F^{ij}(\phi_E)_{ij} = - \langle E, e_k \rangle F^{ij}_{,k} W_{ij}, \quad \nabla \phi_E = W\cdot \langle E, e_k \rangle.
\end{eqnarray*}
\end{proof}

Lemma \ref{lem1} yields immediately Theorem \ref{alex0} in $C^3$ case, which corresponds to the Hartman-Wintner theorem (\cite{HW}).

\begin{cor}\label{C3}
Suppose $f\in C^{2}$ and symmetric, $M_1, M_2$ are two closed convex $C^3$ surfaces satisfy conditions in Theorem \ref{alex0} , then the surfaces are the same up to a translation.
\end{cor}

\begin{proof}
Since $f\in C^2$ is symmetric, $F^{ij}$ in (\ref{Fu}) is in $C^1(\mathbb S^2) \text{ and } u\in C^3(\mathbb S^2)$. By Lemma \ref{lem1} and the strong maximum principle, $X_u$ is a constant vector.
\end{proof}

\bigskip

To precede further, set
\begin{eqnarray*}
\mathcal M=\{p\in \mathbb S^2\ :\ \rho_u(p)=\max_{q\in\mathbb S^2}\rho_u(q)\},\end{eqnarray*}
for each unit vector $E\in \mathbb  R^3$, \begin{eqnarray*}
\mathcal{M}_{E} = \{ p\in \mathbb S^2\ : \ \phi_E(p)= \max_{q\in \mathbb S^2}\phi_E(q)\} .
\end{eqnarray*}

\begin{lem}\label{noisolate}
 $\mathcal M$ and $\mathcal M_E$ have no isolated points.
\end{lem}
\begin{proof}
We prove the lemma for $\mathcal M$, the proof for $\mathcal M_E$ is the same.
If point $p_0\in \mathcal M$ is an isolated point, we may assume $p_0=(0,0,1)$.  Pick $\bar U$ a small open geodesic ball centered at $p_0$ such that $\bar  U$ is properly contained in $\mathbb S^2_{+}$, and pick a sequence of smooth $2-$tensor $(F^{ij}_{\epsilon})>0$ which is convergent to $(F^{ij})$ in $L^{\infty}$-norm in $\bar  U$. Consider
\begin{eqnarray}\label{newDP}
\left\{
\begin{matrix}
F^{ij}_{\epsilon}(u^{\epsilon}_{ij}+u^{\epsilon}\delta_{ij})&=&0  \ \text{ in } \bar  U\\
u^{\epsilon}&=&u \ \text{ on } \partial \bar  U.
\end{matrix}\right.
\end{eqnarray}
Since $x_3>0$ in $\mathbb S^2_{+}$, one may
write $u^{\epsilon}=x_3v^{\epsilon}$ in $\bar  U$. As $(x_3)_{ij}=-x_3 \delta_{ij}$, it easy to check $v^{\epsilon}$ satisfies
\[F^{ij}_{\epsilon}v^{\epsilon}_{ij}+b_kv^{\epsilon}_k= 0, \quad \mbox{in $\bar  U$}.\]
Therefore, (\ref{newDP}) is uniquely solvable.

Since $p_0\in \mathcal M$ is an isolated point, there are open geodesic balls $ \bar  U' \subset \bar  U$ centered at $p_0$ and a small $\delta>0$ such that
\begin{eqnarray}\label{delta}
\rho_u(p_0)- \rho_u(p) \geq \delta  \ \text{ for } \ \forall p\in \partial \bar  U'.
\end{eqnarray}
\par
By the $C^{1,\alpha}$ estimates for linear elliptic equation in dimension two and the uniqueness of the Dirichlet problem (\cite{M, BN2, N2}), $\exists \epsilon_k$  such that
\[
\|u-u^{\epsilon_k}\|_{C^{1,\alpha}( \bar U')}\to 0, \quad \|\rho_u-\rho_{u^{\epsilon_k}}\|_{C^{\alpha}(\bar U')}\to 0.
\]
Together with (\ref{delta}), if $\epsilon_k$ small enough, there is a local maximal point of $\rho_{u^{\epsilon_k}}$ in $\bar U'\subset \bar U$.
Since $u^{\epsilon_k}, F^{ij}_{\epsilon} \in C^{\infty}(\bar U')$ satisfy (\ref{newDP}), it follows from Lemma \ref{lem1} and the strong maximum principle that $\rho_{u^{\epsilon_k}}$ must be constant in $\bar  U'$, $\forall \epsilon_k$ in small enough. This implies $\rho$ is constant in $\bar U'$.  Contradiction.
\end{proof}

\medskip

We now prove Theorem \ref{alex}.

\begin{proof}[\bf Proof of Theorem \ref{alex}.]
\par
For any $p_0\in \mathcal M$, if $\rho_u(p_0)=0$, then $u\equiv 0$. We may assume $\rho_u(p_0)>0$. Set $E:=\frac{X_u(p_0)}{|X_u(p_0)|}$. Choose another two unit constant vectors $\beta_1, \beta_2$ with $<\beta_i, \beta_j> =\delta_{ij}, \beta_i \perp E$ for $i, j=1, 2$. Under this orthogonal coordinates in $\mathbb R^3$,
\begin{eqnarray}\label{dicomp}
X_u(p) = a(p) E + b_1(p)\beta_1+ b_2(p) \beta_2, \ \forall p\in \mathcal M_E.
\end{eqnarray}
On the other hand, $\phi_E(p)=\rho^{1/2}_u(p_0), \forall p\in \mathcal M_{E}$. Thus,
\begin{eqnarray}\label{b=0}
a(p)=\rho^{1/2}_u(p_0), \ b_1(p)=b_2(p)=0, \ \forall p\in \mathcal M_{E}.
\end{eqnarray}
Consider the function $\tilde u(x)=u(x)-\rho_u^{1/2}(p_0)E\cdot x$. (\ref{dicomp}) and (\ref{b=0}) yield, $ \forall p\in \mathcal M_E$,
\begin{eqnarray}\label{nabelu}
\quad \nabla_{e_i}\tilde u(p) = \nabla_{e_i}u(p) -\rho_u^{1/2}(p_0)\langle E,  e_i\rangle=\langle X_u(p), e_i\rangle -\rho_u^{1/2}(p_0)\langle E,  e_i\rangle=0.
\end{eqnarray}
Moreover, $\tilde u(x)$ also satisfies equation (\ref{Fu}).
As pointed out in \cite{BN2}, if $\tilde u$ satisfies an elliptic equation, $\nabla \tilde u$ satisfies an elliptic system of equations. Lemma \ref{noisolate}, (\ref{nabelu}) and the Unique Continuation Theorem of Bers-Nirenberg (P. 13 in \cite{BN1}) imply $\nabla\tilde u\equiv 0$. Thus, $\tilde u(x) \equiv \tilde u(p_0) =0$ and $u(x)$ is a linear function on $\mathbb S^2$.
\end{proof}

\bigskip

Theorem \ref{alex0} is a direct consequence of Theorem \ref{alex}. We now prove Theorem \ref{alexB}.

\begin{proof}[\bf Proof of Theorem \ref{alexB}.]
The main step is to show $u=u^1-u^2\in W^{2,2}(\mathbb S^2)$, using the assumption that $W_{u^l}, l=1,2$ are non-singular Radon measures. It follows from the convexity, the spherical hessians $W_{u^l}, l=1,2$ and $W_u$ are defined almost everywhere on $\mathbb S^2$ (Alexandrov's Theorem). So, we can define $F(W_{u^l}), l=1,2$ almost everywhere in $\mathbb S^2$. As $W_u^{l}, l=1,2$ are nonsingular Radon measures, $W_{u^l}\in L^1(\mathbb S^2)$ (see \cite{EG}), we also have $W_{u}\in L^1(\mathbb S^2)$.
Since $u^1, u^2$ satisfy $F(W_{u^1})= F(W_{u^2})$ for almost every parallel normal $x\in \mathbb S^2$, there is $\Omega\subset \mathbb S^2$ with $|\mathbb S^2\setminus \Omega|=0$, such that $W_u$ satisfies following equation {\it pointwise} in $\Omega$,
\begin{equation*}
 F^{ij}(x) (u_{ij}(x)+u(x)\delta_{ij})=0,  \ x\in \Omega,
\end{equation*}
where $F^{ij}=\int_0^1 \frac{\partial F}{\partial w_{ij}}(tW_u^1+(1-t)W_u^2)d t$.
By Lemma \ref{det}, we can obtain that
\begin{eqnarray*}
|W_u|^2=W_{11}^2+ W_{12}^2 + W_{21}^2 +W_{22}^2 \leq -\frac{2\Lambda}{\lambda}\det W_u, \quad \ x\in \Omega.
\end{eqnarray*}
On the other hand,
\begin{eqnarray*}
\det W_u  \leq \det W_{\tilde{u}},
\end{eqnarray*}
where $\tilde{u}=u^1+u^2$. Thus, to prove $u\in W^{2, 2}(\mathbb S^2)$, it suffices to get an upper bound for $\int_{\mathbb S^2} \det W_{\tilde{u}}$.
\par
Recall that $W_{u^l}\in L^1(\mathbb S^2)$, so $u^l\in W^{2, 1}(\mathbb S^2), l=1, 2$ and the same for $\tilde{u}$. This allows us to choose two sequences of smooth convex bodies $\Omega^l_{\epsilon}$ with supporting functions $u^l_{\epsilon}$ such that $|| \tilde{u}_{\epsilon}- \tilde{u}||_{W^{2, 1}(\mathbb S^2) }\rightarrow 0$ as $\epsilon\rightarrow 0$. By Fatou's Lemma and continuity of the area measures,
\begin{eqnarray*}
\int_{\mathbb S^2} \det W_{\tilde{u}}=\int_{\Omega} \det W_{\tilde{u}} \leq \liminf_{\epsilon\rightarrow 0} \int_{\mathbb S^2} \det W_{\tilde{u}_{\epsilon}} \leq V(\Omega^1)+V(\Omega^2)+2V(\Omega^1, \Omega^2),
\end{eqnarray*}
where $V(\Omega^1), V(\Omega^2)$ denote the volume of the convex bodies $\Omega^1$ and $\Omega^2$ respectively and $V(\Omega^1, \Omega^2)$ is the mixed volume.
\par
It follows that $W_u\in L^2(\mathbb S^2)$ and thus, $u\in W^{2, 2}(\mathbb S^2)$. This implies that $u$ is a $W^{2, 2}$ weak solution of the differential equation
\[
 F^{ij}(x) (u_{ij}(x)+u(x)\delta_{ij})=0, \quad \forall x\in \mathbb S^2.
\]
Finally, the theorem follows directly from Theorem \ref{alex}.
\end{proof}

\bigskip

\begin{rem} Alexanderov proved in \cite{Alex39} that, if $u$ is a homogeneous degree $1$ analytic function in $\mathbb R^3$ with $\nabla^2 u$ definite nowhere, then $u$ is a linear function. As a consequence, Alexandrov proved in \cite{Alex66} that if a analytic closed convex surface in $\mathbb R^3$ satisfying the condition $(\kappa_1-c)(\kappa_2-c)\le 0$ at every point for some constant $c$, then it is a sphere. Martinez-Maure gave a $C^2$ counter-example in \cite{Ma} to this statement, see also \cite{Pa}. The counter-examples in \cite{Ma, Pa} indicate that Theorem \ref{alex} is not true if $F^{ij}$ is merely assumed to be degenerate elliptic. It is an interesting question that under what degeneracy condition on $F^{ij}$ so that Theorem \ref{alex} is still true, even in smooth case. This question is related to similar questions in this nature posted by Alexandrov \cite{Alex56} and Pogorelov \cite{P2}. \end{rem}

\bigskip

\noindent {\it Acknowledgement:} The first author would like to thank Professor Louis Nirenberg for stimulation conversations. Our initial proof was the global maximum principle for $C^3$ surfaces Lemma \ref{lem1} and Corollary \ref{C3} (we only realized the connection of the result of \cite{HNY} to Theorem \ref{alex0} afterward). It was Professor Louis Nirenberg who brought our attention to the paper of \cite{Ma} and suggested using the unique continuation theorem of \cite{BN2}. That leads to Theorem \ref{alexB}. We want to thank him for his encouragement and generosity.

 \end{document}